\documentclass[a4paper]{article}

\pdfoutput=1

\usepackage[utf8x]{inputenc}
\usepackage{amsmath}
\usepackage{hyperref}
\usepackage{amsthm}
\usepackage{amssymb}
\usepackage{faktor}

\usepackage{mathtools}
\usepackage[small]{titlesec}
\usepackage[all,cmtip]{xy}
\usepackage{mycommands}

\numberwithin{equation}{section}

\newtheorem{Thm}{Theorem}[section]
\newtheorem{Prop}[Thm]{Proposition}
\newtheorem{Lem}[Thm]{Lemma}
\newtheorem{Cor}[Thm]{Corollary}

\newtheorem*{Thmintro}{Theorem}
\newtheorem*{Propintro}{Proposition}

\theoremstyle{definition}
\newtheorem*{Def}{Definition}

\newtheorem{Rem}{Remark}

\newcommand{\Oz}{\Omega^2(M,TM)}
\newcommand{\Ozc}{\Omega^2(\Cont,\Cont)}
 \newcommand{\Oee}{\Omega^{1,1}(\Cont,\Cont)}
  \newcommand{\Os}{\Omega^{1,1}_s(\Cont,\Cont)}
  \newcommand{\Oa}{\Omega^{1,1}_a(\Cont,\Cont)}
 \newcommand{\Ozn}{\Omega^{2,0}(\Cont,\Cont)}
 \newcommand{\Onz}{\Omega^{0,2}(\Cont,\Cont)}
\newcommand{\Ot}{\Omega^{3}(M)}
\newcommand{\Otc}{\Omega^{3}(\Cont)}
 \newcommand{\Op}{\Omega^{+}(\Cont)}
 \newcommand{\Om}{\Omega^{-}(\Cont)}
\newcommand{\Ozx}{\Omega^2(\Cont,\Xi)}

 \newcommand{\Oep}{\Omega^1_+(\Cont,\Cont)}
 \newcommand{\Oem}{\Omega^1_-(\Cont,\Cont)}

 \newcommand{\Endcp}{\operatorname{End}_+(\Cont)}
 \newcommand{\Endcm}{\operatorname{End}_-(\Cont)}

\newcommand{\grad}{\operatorname{grad}}
\newcommand{\dom}{\operatorname{dom}}

\titleformat*{\section}{\large \sffamily \bfseries}
\titleformat*{\subsection}{\sffamily \bfseries}

\begin{document}
\title{Adapted connections on metric contact manifolds}
\author{Christoph Stadtm\"uller\footnote{\textsc{Humboldt-Universit\"at zu Berlin}, Institut f\"ur Mathematik, Unter den Linden 6, 10099 Berlin. \texttt{stadtmue@math.hu-berlin.de}}}

\maketitle

\begin{abstract}
In this paper, we describe the space of adapted connections on a metric contact manifold through the space of their torsion tensors. The torsion tensor is an element of $\Omega^2(M,TM)$ which splits into various subspaces. We study the parts of the torsion tensor according to this splitting to completely describe the space of adapted connections. We use this description to obtain characterizations of the generalized Tanka-Webster connection and to describe the Dirac operators of adapted connections.
\end{abstract}

\tableofcontents

\section{Introduction}
Suppose $(M^{2m+1},g,\eta,J)$ is a metric contact manifold, that is $(M,g)$ is a Riemannian manifold, $\eta\in\Omega^1(M)$ and $J\in\End(TM)$ and the various objects are compatible in the following sense:
\begin{equation*}
 J^2X=-X+\eta(X)\xi\quad\text{and}\quad g(JX,Y)=d\eta (X,Y),
\end{equation*}
where $\Cont=\ker\eta$ is called the \emph{contact distribution} and $\xi=\eta^{\natural}$ is called the \emph{Reeb vector field} and satisfies $\eta(\xi)=1,\;\xi\lrcorner d\eta =0$. A connection $\nabla$ is called \emph{adapted} to this structure, if it satisfies
\begin{equation*}
 \nabla g=0,\quad\nabla \eta=0,\quad \nabla\xi=0,\quad \nabla J=0.
\end{equation*}
Due to the conditions $\nabla g=0,\nabla J=0$ (which are actually sufficient), these connections are closely related to hermitian connections on almost hermitian manifolds. These connections have been extensively studied (cf, amongst others, \cite{PL}, \cite{Lichn}) and a detailed description of them using their torsion has been given by Paul Gauduchon \cite{Gau}.

Metric contact manifolds together with their adapted connections are an example of geometries with torsion (cf. \cite{Srni} for an introduction), which have been extensively studied for quite some time and some effort has been made to understand such connections and their Dirac operators, particularly in the case of totally skew-symmetric torsion (see for example \cite{FI}, where the existence of such connections is also discussed for almost metric contact manifolds). However, in general, adapted connections on metric contact manifolds remain largely unexplored. Some contributions have been made by Nicolaescu \cite{Ni} who constructs some adapted connections and Puhle \cite{Puhle} who considers adapted connections on 5-dimensional almost contact metric manifolds. In this paper, we use Gauduchon's methods for almost-hermitian manifolds to study adapted connections on metric contact manifolds. That is, we describe connections through their torsion by means of a careful study of the possible torsion tensors. In order to do so, we study a decomposition of $\Omega^2(M,TM)$ and the parts of the torsion with respect to this splitting.

We obtain a characterization of the torsion tensor, where certain parts are determined independently of the connection $\nabla$ and others may be chosen freely and any such choice will always give the torsion tensor of an adapted connection. More precisely we have (cf. theorem \ref{thm:torsion}):
\begin{Thmintro}
 Let $(M,g,\eta,J)$ be a metric contact manifold $\nabla$ an adapted connection. Then its torsion tensor has the following form:
\begin{equation*}
 T=N^{0,2}+\tfrac{9}{8}\omega-\tfrac{3}{8}\M\omega+B+\xi\otimes d\eta-\tfrac{1}{2}\eta\wedge (J\J) +\eta\wedge \Phi,
\end{equation*}
where $\omega$ is a three-form whose decomposition into $(p,q)$-forms consists only of forms of type $(2,1)$ and (1,2), $B\in \Omega^2(\Cont,\Cont)$ satisfies $B(J\cdot,J\cdot)=B$ and vanishes under the Bianchi operator and $\Phi$ is a skew-symmetric endomorphism of $\Cont$ satisfying $\Phi J=J \Phi$. The other parts are completely determined by the geometry of the contact structure.

Conversely, given any $\omega,B,\Phi$ as above there exists exactly one adapted connection $\nabla(\omega,B,\Phi)$ whose torsion is as given above.
\end{Thmintro}

As an application, we obtain a characterization of the Tanaka-Webster connection and its generalization as the ``simplest'' adapted connection possible. Furthermore, we study the Dirac operators associated to adapted connections and determine those that are formally self-adjoint and the connections that have the same Dirac operator as the Levi-Civit\`{a} connection and the Tanaka-Webster connection. Concerning the Dirac operators, we obtain the following results (cf. proposition \ref{prop:Dirac} and corollaries following it):
\begin{Propintro}
 Let $\nabla(\omega,B,\Phi)$ be an adapted connection as described above. Then the induced Dirac operator has the following properties:
\begin{numlist}
 \item The induced Dirac operator is symmetric if and only if $\tr B=\frac{3}{8}\tr \M\omega$.

 \item Assuming that the induced Dirac operator is symmetric and given a second adapted connection $\nabla(\hat\omega,\hat B,\hat\Phi)$ whose Dirac operator is also symmetric, the two Dirac operators coincide if and only if $\omega=\hat\omega$ and $\Phi=\hat\Phi$

 \item The Dirac operator of $\nabla(\omega,B,\Phi)$ coincides with the one of the Tanaka-Webster connection if and only if $\omega=0$, $\Phi=0$ and $\tr B=0$.
\end{numlist}

\end{Propintro}

The exposition is organized as follows: We begin with a section introducing the basic differential geometric objects on metric contact 
manifolds. In the following section, we carefully study the space of $TM$-valued two-forms on such a manifold, which can be decomposed into various subspaces and apply this theory to certain forms naturally associated to a metric contact manifold. In a fourth section, we then study adapted connections by applying the theory of the preceding section to the space of possible torsion tensors. The final section is dedicated to the study of the Dirac operators associated to adapted connections.

%
\section{Contact and CR structures}

In this section, we give a short introduction to metric contact  and CR manifolds and introduce the basic differential geometric objects one usually considers on such manifolds. Should the reader be interested in more details, we refer him to \cite{Blair}, which offers a comprehensive introduction to contact structures and to \cite{DT} for a comprehensive treatment of CR manifolds.

\subsection{Metric contact manifolds}

Contact manifolds can be viewed as an odd-dimensional analogue of symplectic manifolds and are defined as an odd-dimensional manifold $M^{2m+1}$ carrying a one-form $\eta$ such that $\eta\wedge (d\eta)^m\neq0$, where $(d\eta)^m$ is to be read as taking the wedge product of $d\eta$ with itself $m$ times and $\neq 0$ means nowhere vanishing. Such a contact form induces a distribution $\Cont=\ker \eta$ of rank $2m$. Due to the condition $\eta\wedge (d\eta)^m\neq0$, we deduce from Frobenius' theorem, that this distribution is ``as far from being integrable as possible''.

In order to do geometry on this manifold, we need to issue the manifold with a Riemannian metric which we demand to be compatible (in a sense to be specified) with the contact structure. Moreover, we equip the manifold with an endomorphism $J$ that is also compatible with both the contact and the metric structure. We should note that the concept of a metric contact manifold as defined here is the strongest of a number of concepts relating contact and metric structures, we again refer to the book \cite{Blair} for an introduction.

\begin{Def}
A \emph{metric contact manifold} is a tuple $(M,g,\eta,J)$ with $g$ a Riemannian metric on $M$, $\eta\in\Omega^1(M)$ and $J\in\operatorname{End}(TM)$ such that
\renewcommand{\labelenumi}{(\roman{enumi})}
\begin{enumerate}
\item $\|\eta_x\|=1\quad$ for any $x\in M,$
\item $d\eta(X,Y)=g(JX,Y)\quad$ for any $X,Y\in\VF(M)$ and
\item $J^2=-Id+\eta\otimes \eta^\natural$
\end{enumerate}
\renewcommand{\labelenumi}{(\arabic{enumi})}
\end{Def}

Note that this definition does not explicity require $\eta$ to fulfil the contact condition $\eta\wedge(d\eta)^m\neq 0$. It can, however, be shown that this is indeed the case and that $\xi=\eta^\natural$ is the \emph{Reeb vector field} of the contact form, i.e. it fulfils $\eta(\xi)=1$ and $\xi\lrcorner d\eta=0$. Furthermore, $\xi$ vanishes under $J$, while the \emph{contact distribution} $\Cont=\ker\eta$ is stable under this endomorphism and in fact, $J$ restricted to the contact distribution is an almost-complex structure and thus, in particular, an isomorphism. Because of this almost-complex structure on $\Cont$, we can always chose an \emph{adapted basis} $(e_i,f_i)_{i=1}^m$ of $\Cont$, i.e. an orthonormal basis such that $Je_i=f_i$. Also, the metric $g$ is completely determined by $\eta$ through the equation
\begin{equation*}
 g(X,Y)=d\eta(X,JY)+\eta(X)\eta(Y).
\end{equation*}
Furthermore, we have for any $X,Y\in\Gamma(\Cont)$ that
\begin{equation}\label{eq:etadeta}
 d\eta(X,Y)=X(\eta(Y))-Y(\eta(X))-\eta([X,Y])=-\eta([X,Y]).
\end{equation}
Also, because $\L_\xi \eta=d(\eta(\xi))+\xi\lrcorner d\eta=0$, we obtain that
\begin{equation*}
 0=\L_\xi\eta(X)=\xi(\eta(X))-\eta([\xi,X])
\end{equation*}
and thus that
\begin{equation}\label{eq:etaxikomm}
 \eta([\xi,X])=0\quad\text{for any } X\in\Gamma(\Cont),
\end{equation}
i.e.
\begin{equation}\label{eq:etaxikomm2}
 [\Cont,\xi]\subset \Cont.
\end{equation}

%

\begin{Def}
 The \emph{Nijenhuis tensor} of a metric contact manifold is the skew-symmetric (2,1) tensor given by
\begin{equation*}
 N(X,Y)=[JX,JY]+J^2[X,Y]-J([JX,Y]+[X,JY]).
\end{equation*}
\end{Def}

Note that this differs slightly from the usual definition of Nijenhuis tensors on almost complex manifolds, because here $J^2\neq Id$ in general. For future reference, we state the following two results, where here and in the sequel, $\lc$ denotes the Levi-Civit\`{a} connection:
\begin{Lem}[{\cite[lemmas 6.1 and 6.2\footnote{The tensor $N^{(1)}$ appearing in Blair's book differs slightly form our $N$, but the difference vanishes when taking the product $g(JX,N(\cdot))$}]{Blair}}]\label{lem:J}
 We have the following results on the endomorphism $J$:
 \begin{numlist}
  \item The Levi-Civit\`{a} covariant derivative of $J$ is given by the following formula for any $X,Y,Z\in\VF(M)$:
\begin{equation*}
 2g((\nabla^g_X J)Y,Z)=g(JX,4N(Y,Z))+d\eta(JY,X)\eta(Z)+d\eta(X,JZ)\eta(Y).
\end{equation*}
In particular, $\nabla_\xi J$ vanishes.
  \item The operator $\J=\L_\xi J$ is symmetric (with respect to $g$) and anti-commutes with $J$: $J\J=-\J J$.
 \end{numlist}

\end{Lem}

The operator $J$ gives an almost-complex structure on the \emph{contact distribution} $\Cont=\ker\eta$, i.e. $(J|_{\Cont})^2=-Id_{\Cont}$. Therefore, like for the tangent bundle of an almost complex manifold, the complexifiction of $\Cont$ split into the $\pm i$-eigenspaces of the complex-linearly extended operator $J$, which we shall denote as
\begin{equation*}
 \Cont_c:=\Cont\otimes\C=\Cont^{1,0}\oplus \Cont^{0,1}.
\end{equation*}
Setting $(\Cont^*)^{1,0}=(\Cont^{1,0})^*$ and doing likewise for $(\Cont^*)^{0,1}$, we also obtain a splitting
\begin{equation}
 \Cont_c^*=(\Cont\otimes\C)^*=\Cont^*\otimes \C=(\Cont^*)^{1,0}\oplus (\Cont^*)^{0,1},
\end{equation}
and, taking exterior powers, we obtain the spaces
\begin{equation*}
 \Lambda^{p,q}(\Cont^*):=\Lambda^p\left((\Cont^*)^{1,0}\right)\wedge \Lambda^q \left((\Cont^*)^{0,1}\right),
\end{equation*}
which give us a splitting
\begin{equation*}
 \Lambda^k(\Cont_c^*)=\bigoplus_{p+q=k}\Lambda^{p,q}(\Cont^*).
\end{equation*}
We shall call $\Omega^k_c(\Cont)$ the space of smooth sections of $k$-forms over $\Cont_c^*$ and the smooth sections of the bundles of $(p,q)$-forms $\Omega^{p,q}(\Cont)$.\vspace{10pt}

To conclude this short introduction to contact geometry, we introduce another form that is naturally associated with a metric contact manifold (and modelled on its counterpart from almost hermitian geometry) which shall play an important role lateron.
\begin{Def}\begin{enumerate}
                   \item The \emph{contact Nijenhuis tensor} is the skew-symmetric (2,1)-tensor (i.e. the $TM$-valued two-form) given by
\begin{equation*}
 N(X,Y)=[JX,JY]+J^2[X,Y]-J([JX,Y]+[X,JY]).
\end{equation*}

\item The \emph{K\"ahler form} is the two-form $F\in \Omega^2(M)$ given by
\begin{equation*}
 F(X,Y)=g(JX,Y)=d\eta(X,Y).
\end{equation*}
\end{enumerate}
\end{Def}
Note that, unlike in the almost hermitian case, the K\"ahler form is always closed. This does not have an effect on any kind of integrability of $J$ here, nor does it imply that $N$ vanishes.

\subsection{CR manifolds}

CR manifolds\footnote{What CR stands for is subject of some debate. Some say it means complex real, while others interpret it as Cauchy-Riemann.} are modelled on real hypersurfaces of complex standard space. Let $M^{2m+1}\subset \C^{m+1}$ be such a hypersurface. Then its tangent space is not stable under the complex structure $\widetilde{J}$ of $\C^{m+1}$. Instead, one may consider the space $H_p=T_pM\cap \widetilde{J}(T_pM)$. Then, the bundle $H\subset TM$ is of rank $2m$ and carries an almost-complex structure $J=\widetilde{J}|_{H}:H\rightarrow H$ satisfying the following integrability conditions for all $X,Y\in \Gamma(H)$:
\begin{gather}
 [X,JY]+[JX,Y]\in\Gamma(H),\label{eq:CR1}\\
[JX,JY]-[X,Y]-J([JX,Y]+[X,JY])=0.\label{eq:CR2}
\end{gather}
These properties are used to define an abstract CR manifold:
\begin{Def}
 A \emph{CR manifold} is an odd-dimenisonal manifold $M^{2m+1}$ whose tangent bundles carries a rank $2m$ subbundle $H\subset TM$ equipped with an almost-complex structure $J\colon H\rightarrow H$ satisfying \eqref{eq:CR1} and \eqref{eq:CR2} for all $X,Y\in\Gamma(H)$.
\end{Def}

On any oriented CR manifold, one may find a one-form $\eta\in\Omega^1(M)$ such that $H=\ker \eta$. Note that $\eta$ is not unique, as for any $f\in C^{\infty}(M)$, $f\eta$ will have the same property. Having fixed such a form, we consider the Levy form given by
\begin{equation*}
 L_\eta(X,Y):=d\eta(X,JY)
\end{equation*}
for any $X,Y\in H$. If $L_\eta$ is nodegenerate, then $\eta$ is a contact form. We will, however, concentrate on the case where $L_\eta$ is even positive-definite. In this case, $(M,H,J,\eta)$ is called a \emph{strictly pseudoconvex CR structure} and we can define a Riemannian metric (the \emph{Webster metric}) on $M$ by
\begin{equation*}
 g_\eta=L_\eta+\eta\otimes\eta.
\end{equation*}
Then, $(M,g_\eta,\eta,J)$, where $J$ is extended by $J(\eta^\natural)=0$, is a metric contact manifold.

Conversely, given any metric contact manifold, it is CR (i.e. $(M,\Cont,J_{\Cont})$ is a strictly pseudoconvex CR manifold) if and only if \eqref{eq:CR2} is fulfilled, or alternatively, if $J$ satisfies the following identity:
\begin{equation*}
 J(N(X,Y))=0\quad\text{for any }X,Y\in\Gamma(H).
\end{equation*}

\section{Differential forms on metric contact manifolds}
In this section, we give a description of the spaces of $TM$-valued 2-forms on $M$, which we denote $\Omega^2(M,TM)$, and the space $\Omega^3(M)$ of real-valued 3-forms, by describing how these spaces can be decomposed into subspaces and showing certain relations between these subspaces. As an application, we will study how the Nijenhuis tensor and the covariant derivative of the K\"ahler form behave under this splitting. Before we begin the actual study of these spaces, we quickly introduce some conventions and operators that will be used in the following: For a $TM$-valued two-form $B$, we agree to write
\begin{equation}\label{eq:conv}
 B(X;Y,Z):=g(X,B(Y,Z))\quad\text{for any }X,Y,Z\in TM.
\end{equation}
Conversely, we may understand a three-form $\omega$ as a $TM$-valued two-form via
\begin{equation}\label{eq:threetwo}
 \omega(X,Y,Z)=g(X,\omega(Y,Z)).
\end{equation}
Furthermore, we introduce the following operators: The Bianchi operator
\begin{equation*}
\b\colon \Omega^2(M,TM) \rightarrow \Omega^3(M) 
\end{equation*}
 given by
\begin{equation*}
\b B(X,Y,Z)=\tfrac{1}{3}\left(B(X;Y,Z)+B(Y;Z,X)+B(Z;X,Y)\right),
\end{equation*}
the operator
\begin{align*}
\M\colon \Omega^2(M,TM)&\longrightarrow \Omega^2(M,TM)\\
B&\longmapsto B(J\cdot,J\cdot)
\end{align*}
and the trace operator
\begin{equation*}
 \tr\colon \Omega^2(M,TM)\longrightarrow \Omega^1(M)
\end{equation*}
given, for an ON basis $(b_i)$ of $TM$, by
\begin{equation*}
 \tr B(X)=\sum_{i=1}^{2m+1}B(e_i;e_i,X).
\end{equation*}

Finally, the subspaces we are about to introduce will always  be denoted by a sub- and superscript indices. If we apply the same indices to a form, we mean its part in the respective subspace.

We have now set notation and begin considering the tangent bundle. Denoting $\Xi=\R\xi$, we see that the tangent bundle splits as $TM=\Cont\oplus \Xi$
and thus, we have some induced splittings on the spaces of exterior powers:
\begin{align*}
 TM\otimes\Lambda^2(T^*M)&=\Cont\otimes \Lambda^2(\Cont^*)\;\oplus\; \xi\otimes \Lambda^2(\Cont^*)\;\oplus\; TM\otimes \eta\wedge \Cont^*,\\
 \Lambda^3(T^*M)&=\Lambda^3(\Cont^*)\oplus \eta\wedge \Lambda^2(\Cont^*).
\end{align*}
The theory developed by Paul Gauduchon for the respective forms over an almost hermitian manifold carries over almost word-for-word to the bundles $\Cont\otimes \Lambda^2(\Cont^*)$ and $\Lambda^3(\Cont^*)$. We will review these results and translate them to our case in a first subsection, and deal with the remaining spaces in a second subsection.

%
\subsection{The forms over the contact distribution}\label{sec:formscont}
In this part, we collect the results on the spaces $\Omega^2(\Cont,\Cont)$ and $\Omega^3(\Cont)$. All manipulations we are about to perform on these spaces are pointwise and we will therefore use the bundles and spaces of sections indiscriminately. The calculations on this bundle are nearly equivalent to those on the tangent bundle of an almost-hermitian manifold and thus we simply ``translate'' the results of \cite{Gau} to our case, omitting all proofs as they may be found in the original paper. Alternatively, one finds a detailed expostion in the first chapter of \cite{DA}.

To begin with, we introduce the following subspaces:
\begin{align*}
 \Oee&:=\{B\in \Ozc\;|\;\M B=B\},\\
 \Ozn&:=\{B\in\Ozc\;|\; B(JX,Y)=JB(X,Y)\;\forall X,Y\in \Gamma(\Cont)\},\\
\shortintertext{and}
 \Onz&:=\{B\in\Ozc\;|\; B(JX,Y)=-JB(X,Y)\;\forall X,Y\in \Gamma(\Cont)\}
\end{align*}
and thus obtain the decomposition
\begin{equation*}
 \Ozc=\Oee\oplus\Ozn\oplus\Onz.
\end{equation*}
Given a form $B\in \Oz$, we denote its part in $\Ozc$ as
\begin{equation*}
 B_c=B^{1,1}+B^{2,0}+B^{0,2}.
\end{equation*}
We note that $\Ozn\oplus \Onz$ forms the eigenspace of $\M$ to the eigenvalue $-1$.
The image of $\Ozc$ under $\b$ lies in $\Otc$ and we will now study that space. It can be embedded into the space of complex forms $\Omega^3_c(\Cont)\simeq \Otc\otimes \C$ and thus, any $\omega\in \Otc$ admits a splitting into (complex) forms of type $(p,q)$. We define
\begin{align*}
 \omega^+&:=\omega^{2,1}+\omega^{1,2},\\
 \omega^-&:=\omega^{3,0}+\omega^{0,3}.
\end{align*}
The reason why we consider these forms is that, as opposed to the simple parts of type $(p,q)$, they are again real forms (i.e. real-valued when evaluated on elements of $\Cont$). We define the respective spaces as
\begin{align*}
 \Omega^+(\Cont)&:=\{\omega\in \Otc\;|\;\omega=\omega^+\},\\
 \Omega^-(\Cont)&:=\{\omega\in \Otc\;|\;\omega=\omega^-\}.
\end{align*}

Moreover, we have
\begin{Lem}[{\cite[p.262]{Gau}}]
 Let $\omega\in \Otc$. We also consider $\omega$ as an element of $\Ozc$ via equation \eqref{eq:threetwo} and it thus admits a splitting as $\omega=\omega^{1,1}+\omega^{2,0}+\omega^{0,2}$. Then the following relations are satisfied:
\begin{align*}
 \omega^+&=\omega^{2,0}+\omega^{1,1},&
 \omega^{2,0}&=\tfrac{1}{2}\left(\omega^+-\M\omega^+\right),\\
 \omega^-&=\omega^{0,2},&\omega^{1,1}&=\tfrac{1}{2}\left(\omega^+ +\M\omega^+\right).
\end{align*}
\end{Lem}

Furthermore, for an element of any of the subspaces of $\Ozc$, we can determine the type of its image under the Bianchi operator as the following lemma states more precisely:
\begin{Lem}[{\cite[section 1.4]{Gau}}]\label{lem:splittingident}\ \\[-15pt]
\begin{numlist}
  \item Let $B\in \Onz$. then $\b B\in \Om$.

 \item For any $B\in \Ozn$, we have $\b B\in \Op$. Moreover $\b|_{\Omega^{2,0}}\colon \Ozn\rightarrow \Op$ is an isomorphism and its inverse is given by
\begin{equation}\label{eq:bOzn}
(\b|_{\Omega^{2,0}})^{-1}\omega=\tfrac{3}{2}\left(\omega-\M\omega\right).
\end{equation}

 \item Let $\Os$ be the subspace of $\Oee$ of elements vanishing under $\b$ and $\Oa$ its orthogonal (with repect to the metric $g$ extended to forms in the usual way) complement. Then, $\b|_{\Omega^{1,1}_a}:\Oa\rightarrow \Op$ is an isomorphism with its inverse given by
\begin{equation} \label{eq:bOa}
(\b|_{\Omega^{1,1}_a})^{-1}(\omega)=\tfrac{3}{4}\left(\omega+\M\omega\right).
\end{equation}

  \item Combining the above results, we see that for any $B\in \Ozc$ we have $(\b B)^-=\b(B^ {0,2})$ and $(\b B)^+=\b(B^{1,1}+B^{2,0})$. Furthermore, we obtain an isomorphism $\phi\colon \Ozn\rightarrow \Oa$ given by
\begin{equation*}
 \phi(B)=\tfrac{3}{4}(\b B+\M \b B)\quad\text{and}\quad\phi^{-1}(A)=\tfrac{3}{2}(\b A-\M\b A)
\end{equation*}
 \end{numlist}
\end{Lem}
Finally, a calculation verifies that for any $\omega\in \Op$, the following identity is satisfied
\begin{equation}\label{eq:bM}
 \b \M \omega=\frac{1}{3}\omega.
\end{equation}

\begin{Rem}\label{rem:formsthree} \textbf{The case of a 3-manifold}
 
In the case of a metric contact 3-manifold ($m=1$), the space $\Omega^3(\Cont)$ vanishes. Furthermore, using a local adapted basis $(e_1,f_1)$ of $\Cont$, the space of $TM$-valued two-forms is locally spanned by $e_1\otimes e^1\wedge f^1$ and $f_1\otimes e^1\wedge f^1$. These forms are of type (1,1) vanish under $\b$ and have trace $f^1$ and $-e^1$ respectively.
\end{Rem}

We now have all the links between the various subspaces of $\Ozc$ and $\Ot$ needed and conlude this part, turning next to the forms that do not take their arguments exclusively in $\Cont$.

%
\subsection{The other parts}
What is left to consider now are the parts of $\Oz$ for  which $\xi$ may appear as an argument or a value. First, we consider the elements of $\Ozx$: Any element of this space has the form $\xi\otimes \alpha$, where $\alpha\in \Omega^2(\Cont)$. Therefore, its image under $\b$ is obviously given by
\begin{equation}
\b (\xi\otimes \alpha)=\frac{1}{3}\eta\wedge\alpha\, \in\, \eta\wedge\Omega^2(\Cont) \label{eq:3bx}
\end{equation}
We can decompose $\Omega^2(\Cont)$ as
\begin{gather*}
 \Omega^2(\Cont)=\Omega^2_+(\Cont)\oplus \Omega^2_-(\Cont),\\
\shortintertext{where}
 \Omega^{2}_{\pm}(\Cont)=\{\alpha\in \Omega^2(\Cont)\;|\; \alpha(J\cdot,J\cdot)=\pm\alpha\}.
\end{gather*}
These spaces are again the eigenspaces of the involution $\M$ (defined on $\Omega^2(\Cont)$ just as before) to the eigenvalues 1 and $-1$.
This may be regarded as a decomposition of $\Ozx$ and then, by \eqref{eq:3bx}, is stable under $\b$. 

Finally, there remains a last part to be considered, the forms in 
\begin{equation*}
TM\otimes \eta\wedge \Cont^*=\Cont\otimes \eta\wedge \Cont^*\;\oplus\;\xi\otimes \eta\wedge \Cont^*.
\end{equation*}
Any element of $\Cont\otimes \eta\wedge \Cont^*$ may be interpreted as $\eta\wedge \Phi$, where $\Phi$ is an endomorphism of $\Cont$ and we understand this wedge product to mean
\begin{equation*}
\eta\wedge \Phi(X;Y,Z)=\eta(Y)g(X,\Phi(Z))-\eta(Z)g(X,\Phi(Y)) 
\end{equation*}
 for any $X,Y,Z\in\VF(M)$, where we extend $\Phi$ by $\Phi\xi=0$. Then, $\Phi$ may be further decomposed with respect to its behaviour with respect to $g$ and $J$. We write
\begin{align*}
 \End_\pm(\Cont)&:=\{F:\Cont\to\Cont\;|\; g(X,FY)=\pm g(FX,Y)\},\\
 \End^J_\pm(\Cont)&:=\{F\in\End_\pm(\Cont)\;|\;FJ=JF\}.
\end{align*}

 The behaviour of the symmetric and skew-symmetric parts under the Bianchi operator is described in the following lemma:
\begin{Lem}\label{lem:bEnd}
 Let $\Phi\in \End_+(\Cont)$ and $\Psi\in\End_-(\Cont)$. Then, we have
\begin{equation*}
 \b(\eta\wedge \Phi)=0\quad\text{and}\quad \b(\eta\wedge \Psi)(\xi,X,Y)=\tfrac{2}{3}g(Y,\Psi X).
\end{equation*}
\end{Lem}
\begin{proof}
 For any $F\in \End(\Cont)$, we have
\begin{align*}
 3\b(\eta\wedge F)(\xi,X,Y)&=\eta\wedge F(\xi;X,Y)+\eta\wedge F(X;Y,\xi)+\eta\wedge F(Y;\xi,X)\\
&=-g(X,F Y)+g(Y, FX).
\end{align*}
Then, using symmetry and skew-symmetry respectively yields the claim.
\end{proof}

Summing up the various decompositions we have introduced above (and considering sections), we have the following decomposition for any element $B\in \Oz$:
\begin{equation}\label{eq:decomp}
 B=B^{2,0}+B^{1,1}+B^{0,2}+\xi\otimes B^2_++\xi\otimes B^2_-+\eta\wedge B^{1}_+ +\eta\wedge B_-^1 + \xi\otimes \eta\wedge B^1_\R,
\end{equation}
where $B^1_+\in\Oep\simeq \Endcp$ is a symmetric endomorphism, $B^1_-\in\Oem\simeq\Endcm$ a skew-symmetric one and $B^1_\R\in \Omega^1(\Cont)$. We will sometimes group these parts as follows:
\begin{align*}
 B^2&:=B^2_++B^2_-,\\
 B^1&:=B^1_++B^1_-+B^1_\R.
\end{align*}

%
\subsection{Application: K\"ahler form and Nijenhuis tensor}
In this section, we study the Nijenhuis tensor $N$ and the (Levi-Civit\`{a}) covariant derivative of the K\"ahler form $\lc F$, which we consider as an element of $\Oz$ via the conventions $(\lc F)(X;Y,Z)=(\lc_XF)(Y,Z)$ and \eqref{eq:conv}, and determine their parts according to the above decomposition (see also \cite[proposition 1]{Gau} for the almost-hermitian model and \cite[pp 366f]{Ni} for the Nijenhuis tensor).

\begin{Prop}\label{prop:NF}
 The Nijenhuis tensor of a metric contact manifold has the following properties:
\renewcommand{\labelenumi}{{\upshape(N\arabic{enumi})}}
\begin{enumerate}
 \item We have $N=N^{0,2}-\frac{1}{4}\xi\otimes d\eta-\tfrac{1}{4}\eta\wedge (J\J)$, where we recall that $\J=\L_\xi J$.

 \item $N$ is trace-free.
 
 \item $N^{0,2}$ vanishes under $\b$.
\end{enumerate}
Furthermore, for $\lc F$, we have the following properties:
\renewcommand{\labelenumi}{{\upshape (F\arabic{enumi})}}
\begin{enumerate}
 \item The following parts of $\lc F$ vanish:
\begin{equation*}
 (\lc F)^{1,1}\equiv 0,\quad (\lc F)^{2,0}\equiv0\quad\text{and}\quad(\lc_\xi F)\equiv 0.
\end{equation*}

 \item $(\lc F)^{0,2}$ and $N^{0,2}$ determine each other via
\begin{equation*}
 (\lc F)^{0,2}(X;Y,Z)=2N^{0,2}(JX,Y,Z)
\end{equation*}
for any $X,Y,Z\in \Gamma(\Cont)$.

 \item $(\lc F)^1\in \Omega^1(\Cont,\Cont)$ and it is given by
\begin{align*}
 (\lc F)^1 X&=2N^1(JX)+\tfrac{1}{2}JX,\\
\shortintertext{or, alternatively, by}
g((\lc F)^1 X,Y)&=g(JY,2N^1(X))+\tfrac{1}{2}d\eta(X,JY).
\end{align*}

\item Altogether, $(\lc F)$ has the following form:
\begin{equation*}
 \lc F=2N^{0,2}(J\cdot;\cdot,\cdot)+\eta\wedge (\lc F)^1.
\end{equation*}
\end{enumerate}
\end{Prop}
\begin{proof}
\begin{description}
 \item[(N1) and (N2)] For $Y,Z\in\Gamma(\Cont)$, we have that
\begin{align*}
4N(JY,Z)&=[J^2Y,JZ]+J^2[JY,Z]-J([J^2Y,Z]+[JY,JZ])\\
&=-J[JY,JZ]+J[Y,Z]-[Y,JZ]+J^2[JY,Z]
\end{align*}
Now, for $X,Y,Z\in \Gamma(\Cont)$, this implies $g(X,N(JY,Z))=g(JX,N(Y,Z))$. This implies that the part in $\Ozc$ is of type $(0,2)$. Furthermore, because $J(TM)\subset \Cont$, we have for $Y,Z\in \Gamma(\Cont)$ that
\begin{align*}
 4g(\xi,N(Y,Z))=g(\xi,[JY,JZ])&=\eta([JY,JZ])\\
&=-d\eta(JY,JZ).
\end{align*}
The explicit form of $N^1$ is an easy calculation and that it is symmetric follows by the symmetry of $\J$ and the fact that $J\circ\J=-\J\circ J$ (cf lemma \ref{lem:J}).
(N2) follows immediately.

\item[(F1)] Let $X,Y,Z\in\Gamma(\Cont)$. Then, we have
\begin{align*}
(\lc F)(X;Y,Z)&=(\lc_X F)(Y,Z)\\
&=X(F(Y,Z))-F(\lc_X Y,Z)-F(Y,\lc_X Z)\\
&=X(g(JY,Z))+g(\lc_X Y,JZ)-g(JY,\lc_X Z)\\
&=-X(g(Y,JZ))+X(g(Y,JZ))-g(Y,\lc_X JZ)\\&\phantom{=}\;-X(g(JY,Z))+g(\lc_X JY,Z)\\
&=X(g(Y,JZ))-g(Y,\lc_X JZ)+g(J(\lc_X JY),JZ)\\
&=-X(F(JY,JZ))+F(\lc_X JY,JZ)+F(JY,\lc_X JZ)\\
&=-(\lc F)(X;JY,JZ).
\end{align*}
Thus, $(\lc F)^{1,1}=0$. Concerning $(\lc_\xi F)$, we use that $\lc_\xi J=0$ (lemma \ref{lem:J}) to obtain
\begin{align*}
 \lc_\xi F(X,Y)&=\xi(F(X,Y))-F(\lc_\xi X,Y)-F(X,\lc_\xi Y)\\
&=g(\lc_\xi (JX),Y)+g(JX,\lc_\xi Y)-g(J\lc_\xi X,Y)-g(JX,\lc_\xi Y)\\
&=0.
\end{align*}
Furthermore, by a well-known formula for the exterior derivative, we have for any $X,Y,Z\in\Gamma(\Cont)$ that
\begin{align*}
 0=dF(X,Y,Z)&=(\lc_X F)(Y,Z)-(\lc_Y F)(X,Z)+(\lc_Z F)(X,Y)\\
&=3\b(\lc F)(X,Y,Z),
\end{align*}
i.e. $\b(\lc F)=0$. Now, using \eqref{eq:bOzn}, we deduce that
\begin{equation*}
 (\lc F)^{2,0}=\tfrac{3}{2}\left((\b(\lc F))^+ +\M(\b(\lc F))^+\right)=0.
\end{equation*}
This concludes the proof of (F1).

\item[(F2) and (N3)] Explicitly writing out $N$ and then using that $\lc$ is torsion-free and metric, we obtain for any $X,Y,Z\in \Gamma(\Cont)$:
\begin{multline}\label{eq:F2}
 4(N(JX;Y,Z)+N(JY;X,Z)-N(JZ;X,Y))=\\2(-g(JZ,\lc_{JX}JY)+g(Z,\lc_{JX}Y)+g(JZ,\lc_X Y)+g(Z,\lc_X JY)).
\end{multline}
Note that we could write $N^{0,2}$ instead of $N$ here as all other parts vanish for arguments in $\Cont$.
On the other hand, consider $(\lc F)$. We know that $(\lc F)^{1,1}$ and $(\lc F)^{2,0}$ vanish. Thus, by the properties of $(0,2)$-forms, we obtain
\begin{align*}
 2(\lc F)^{0,2}(X;JU,Z)&=(\lc F)(X;JU,Z)+(\lc F)(JX;U,Z)\\
&=-g(JZ,\lc_{JX}JU)+g(Z,\lc_{JX}U)+g(JZ,\lc_X U)\\&\phantom{=}\;+g(Z,\lc_X JU).
\end{align*}
Substituting $Y=JU$ and comparing this with \eqref{eq:F2} yields
\begin{equation}\label{eq:F2b}
 (\lc F^{0,2})(X;Y,Z)=N^{0,2}(JX;Y,Z)+N^{0,2}(JY;X,Z)-N^{0,2}(JZ;X,Y)
\end{equation}
Using this, we obtain that
\begin{equation*}
 \b N^{0,2}(X,Y,Z)=-\b(\lc F)^{0,2}(JX;Y,Z)=0,
\end{equation*}
which proves (N3). Finally, using \eqref{eq:F2b} we conclude that
\begin{align*}
 (\lc F)^{0,2}(X;Y,Z)&=N^{0,2}(JX;Y,Z)+N^{0,2}(JY;X,Z)-N^{0,2}(JZ;X,Y)\\
&=-3(\b N^{0,2})(JX;Y,Z)+2N^{0,2}(JX,Y,Z).
\end{align*}
Using that $\b N^{0,2}$ vanishes, this yields (F2).

\item[(F3)] We have for any $X\in \Gamma(\Cont)$ and $Y\in \VF(M)$ that
\begin{align*}
 \eta\wedge(\lc F)^1(Y;\xi,X)&=(\lc_Y F)(\xi,X)\\
&=Y(F(\xi,X))-F(\lc_Y \xi,X)-F(\xi, \lc_Y X)\\
&=g(\lc_Y \xi,JX).
\end{align*}
Then, using lemma \ref{lem:J}, we deduce
\begin{align}
 2g((\lc F)^1 X,Y)&=-2g(J\lc_Y \xi,X)=2g((\lc_Y J)\xi,X)\notag\\
&=g(JY,4N(\xi,X))+d\eta(J\xi,Y)\eta(X)+d\eta(Y,JX)\eta(\xi)\notag\\
&=g(JY,4N^1(X))+d\eta(Y,JX)\label{eq:F1first}\\
&=-g(Y,4JN^1(X))+g(JY,JX)\notag\\
&=g(Y,4N^1(JX))+g(Y,X).
\end{align}
Now, \eqref{eq:F1first} proves the second identity in (F3) and the last of the above equations the first one.
\end{description}
\end{proof}

\section{Adapted connections}
\subsection{Definition and basic properties}

We begin by introducing adapted connections and discussing some basic properties. A connections is called adapted if it parallelizes the metric contact structure, more precisely:

\begin{Def}
 Let $(M,g,\eta,J)$ be a metric contact manifold. Then, a connection $\nabla$ is called \emph{adapted} if it is metric and satisfies
\begin{equation*}
 \nabla J=0,\quad \nabla \eta=0\quad\text{and}\quad \nabla \xi=0.
\end{equation*}
\end{Def}

In fact, this definition is redundant, as the following lemma shows:
\begin{Lem}\label{lem:adap}
 Let $(M,g,\eta,J)$ be a metric contact manifold.
\begin{numlist}
 \item Let $\nabla$ be a metric connection such that $\nabla J=0$. Then $\nabla$ is adpated.
 \item Let $\nabla$ be adapted. Then, for any $X\in \VF(M)$ and $Y\in\Gamma(\Cont)$, the vector field $\nabla_X Y$ is again in $\Gamma(\Cont)$.
\end{numlist}
\end{Lem}
\begin{proof}
 For (1), we only need to show that $\nabla \xi=0$. That $\nabla \eta=0$ is then immediate. We know that $0=(\nabla J)\xi=\nabla (J\xi)-J(\nabla \xi)$. Because $J\xi=0$, this implies $J(\nabla \xi)=0$, i.e. $\nabla \xi=\lambda \xi$, with $\lambda \in C^\infty(M)$. However, because $\xi$ has constant length, $g(\nabla \xi,\xi)=0$ and thus $\lambda \equiv 0$.

 For (2), we then obtain $g(\xi,\nabla_X Y)=X(g(\xi,Y))-g(\nabla_X \xi,Y)=0$.
\end{proof}

Little is known about these connections so far, the most well-known example is the Tanaka-Webster connection in the case where the metric contact structure is induced by a stricly pseudoconvex CR manifold. It is defined by demanding that it be metric and explicitly giving its torsion. A generalization of this connection to arbitrary metric contact manifolds has been constructed by Tanno \cite{Tanno}, which is, however, in general not adapted. Nicolaescu \cite{Ni} has constructed a different generalization, which is indeed adapted and another adapted connection which induces the same Dirac operator as $\lc$. We shall return to these connections later.

\subsection{The torsion tensor of an adapted connection}

To any metric connection, we can associate two tensors, the \emph{torsion tensor}
\begin{equation*}
 T\in \Omega^2(M,TM)\quad\text{given by}\quad T(X,Y)=\nabla_X Y-\nabla_Y X-[X,Y]
\end{equation*}
and the \emph{potential}
\begin{equation*}
 A\in \Omega^1(M,End_-(TM))\quad\text{given by}\quad A_X Y=\nabla_X Y-\lc_X Y.
\end{equation*}
We can consider $A$ as a $TM$-valued two form, via
\begin{equation*}
 A(X;Y,Z)=g(A_X Y,Z)
\end{equation*}
with the usual conventions. Then, torsion and potential are related via
\begin{align*}
T&=-A+3\b A,\\
A&=-T+\frac{3}{2}\b T.
\end{align*}
Thus, any metric connections is completely determined by its torsion, i.e. any 2-form is the torsion tensor of a metric connection. In order to obtain an adapted connection, we need to impose additional restrictions. To this end, we study the various parts of the torsion tensor in the following theorem (cf \cite[proposition 2]{Gau} for the hermitian model).
\begin{Thm}\label{thm:torsion}
 Let $(M,g,\eta,J)$ be a metric contact manifold and $\nabla$ an adapted connection. Then, its torsion tensor $T$ has the following properties:
\begin{numlist}
 \item The $(0,2)$-part is given by $T^{0,2}=N^{0,2}$, i.e. in particular independent of $\nabla$.

 \item The following relationships are statisfied for the parts of type $(2,0)$ and $(1,1)$:
\begin{gather*}
 T^{2,0}-\phi^{-1}(T^{1,1}_a)=0,\\
\shortintertext{or, equivalently,}
\b(T^{2,0}-T^{1,1}_a)=0
\end{gather*}

 \item The part in $\Ozx$ is independent of $\nabla$ and given by
\begin{equation*}
 T^2=T^2_+=d\eta.
\end{equation*}

 \item We have the following results on the endomorphism $T^1$. Its symmetric part $T^1_+$ is independent of $\nabla$ and given by
\begin{equation*}
 T^1_+=-\tfrac{1}{2}J\J,
\end{equation*}
where we recall $\J=\L_\xi J$, while the skew-symmetric part $T^1_-$ lies in $\End_-^J(\Cont)$.
 
 \item The part $T^1_\R$ vanishes.
\end{numlist}
Conversely, for any $\omega\in \Op$, $B\in \Os$ and $\Phi\in \End_-^J(\Cont)$, there exists an adapted connection, whose torsion tensor satisfies
\begin{equation*}
 (\b T)^+=\omega,\quad T^{1,1}_s=B\quad\text{and }T^1_-=\Phi.
\end{equation*}
The total torsion tensor then has the following form:
\begin{equation}\label{eq:torsion}
 T=N^{0,2}+\tfrac{9}{8}\omega-\tfrac{3}{8}\M\omega+B+\xi\otimes d\eta-\tfrac{1}{2}\eta\wedge (J\J) +\eta\wedge \Phi.
\end{equation}
\end{Thm}

In the sequel, we will denote the adapted connection defined by $\omega$, $B$ and $\Phi$ as $\nabla(\omega,B,\Phi)$.

\begin{proof}
 \textbf{First step:} We prove that $\nabla$ is adapted if and only if it satisfies
\begin{equation}\label{eq:Aadap}
 A(X;Y,JZ)+A(X;JY,Z)=-(\lc F)(X;Y,Z)
\end{equation}
for any $X,Y,Z\in\VF(M)$. To this end, we compute
\begin{align*}
 A(X;Y,JZ)+A(X;JY,Z)&=g(\nabla_X Y-\lc_X Y,JZ)+g(\nabla_X(JY)-\lc_X(JY),Z)\\
&=g((\nabla_X J)Y,Z)-g(\lc_X Y,JZ)-g(\lc_X(JY),Z).
\end{align*}
On the other hand we have that
\begin{align*}
 -(\lc F)(X;Y,Z)&=-X(F(Y;Z))+F(\lc_X Y,Z)+F(Y, \lc_X Z)\\
&=-g(\lc_X(JY),Z)-g(JY,\lc_X Z)-g(\lc_X Y,JZ)\\
&\phantom{=}+g(JY, \lc_X Z).
\end{align*}
This yields the claimed equivalence.

\textbf{Second step:} Using that $A=-T+\frac{3}{2}\b T$, we deduce that \eqref{eq:Aadap} is equivalent to
\begin{equation*}
 T(X;Y,JZ)+T(X;JY,Z)-\tfrac{3}{2}(\b T(X;Y,JZ)+\b T(X;JY,Z))=(\lc F)(X;Y,Z)
\end{equation*}
for all $X,Y,Z\in\VF(M)$. Alternatively, using that
\begin{equation*}
T(\xi;\xi,X)=g(\xi,\nabla_\xi X-\nabla_X \xi-[\xi,X])=-\eta([\xi,X])\stackrel{\eqref{eq:etaxikomm}}{=}0,
\end{equation*}
which proves (5), this may be written as the system of equations
\begin{align}
 T(X;Y,JZ)+T(X;JY,Z)-\tfrac{3}{2}(\b T(X;Y,JZ)+&\b T(X;JY,Z))\notag\\&=(\lc F)(X;Y,Z),\label{eq:Tcont}\\
T(\xi;Y,JZ)+T(\xi;JY,Z)-\tfrac{3}{2}(\b T(\xi;Y,JZ)+&\b T(\xi;JY,Z))\notag\\&=0,\label{eq:T2}\\
T(X;\xi,JZ)-\tfrac{3}{2}\b T(X;\xi,JZ)&=(\lc F)(X;\xi,Z)\label{eq:T1}
\end{align}
for any $X,Y,Z\in\Gamma(\Cont)$.
Furthermore, using the results of section \ref{sec:formscont}, we find that \eqref{eq:Tcont} is equivalent to the system
\begin{gather}
 2T^{0,2}(JX;Y,Z)-3(\b T)^-(JX;Y,Z)=(\lc F)^{0,2}(X;Y,Z),\label{eq:Tcont1}\\
 -T^{2,0}(JX;Y,Z)-\tfrac{3}{2}((\b T)^+(X;Y,JZ)+(\b T)^+(X;JY,Z))=0\label{eq:Tcont2}.
\end{gather}

\textbf{Third step:} We now prove the claims (1)-(4). To begin with, we obtain from \eqref{eq:Tcont1} and proposition \ref{prop:NF} that
\begin{equation}
 2T^{0,2}(JX;Y,Z)-3(\b T)^-(JX;Y,Z)=N^{0,2}(X;Y,Z)\label{eq:Tcont1a}
\end{equation}
Furthermore, we use a well-known formula for the exterior derivative and see that
\begin{align*}
 0=dF(X,Y,Z)
&=X(g(JY,Z))-Y(g(JX,Z))+Z(g(JX,Y))\\
&\phantom{=}-g(J[X,Y],Z)+g(J[X,Z],Y)-g(J[Y,Z],X).
\end{align*}
Because $\nabla$ is metric and by the definition of $T$, this can be seen to be equal to
\begin{align*}
0&=g(\nabla_X JY,Z)+g(JY,\nabla_X Z)-g(\nabla_Y JX,Z)-g(JX,\nabla_Y Z)
+g(\nabla_Z JX,Y)\\
&\phantom{=}+g(JX,\nabla_Z Y)+g(\nabla_X Y,JZ)-g(\nabla_Y X,JZ)-T(JZ;X,Y)-g(\nabla_X Z,JY)\\
&\phantom{=}+g(\nabla_Z X,JY)+T(JY;X,Z)+g(\nabla_Y Z,JX)-g(\nabla_Z Y,JX)-T(JX;Y,Z)\\
&=g(\nabla_X JY,Z)-g(\nabla_Y JX,Z)+g(\nabla_Z JX,Y)+g(\nabla_X Y,JZ)-g(\nabla_Y X,JZ)\\
&\phantom{=}-T(JZ;X,Y)+g(\nabla_Z X,JY)+T(JY;X,Z)-T(JX;Y,Z).
\end{align*}

Using that $\nabla J=0$, we then obtain
\begin{align*}
 0&=g(J(\nabla_X Y),Z)-g(J(\nabla_Y X),Z)+g(J(\nabla_Z X),Y)-g(J(\nabla_X Y),Z)\\
 &\phantom{=}+g(J(\nabla_Y X),Z)
-T(JZ;X,Y)-g(J(\nabla_Z X),Y)+T(JY;X,Z)\\
&\phantom{=}-T(JX;Y,Z)\\
&=-T(JZ;X,Y)+T(JY;X,Z)-T(JX;Y,Z).
\end{align*}
Taking the (0,2)-part on both sides, we see that
\begin{align*}
 0=-3\b T^{0,2}(JX;Y,Z)=-3(\b T)^-(JX;Y,Z).
\end{align*}
Inserting this into \eqref{eq:Tcont1a} yields (1).

Next, we use that $(\b T)^+=\b(T^{1,1}_a+T^{2,0})$ to deduce from \eqref{eq:Tcont2} that
\begin{align*}
 T^{2,0}(X;Y,Z)=&\tfrac{3}{4}(\b T^{1,1}_a-\M(\b T^{1,1}_a))(JX;JY,Z)\\
&+\tfrac{3}{4}(\b T^{2,0}-\M(\b T^{2,0}))(JX;JY,Z).
\end{align*}
Using lemma \ref{lem:splittingident}, we obtain that
\begin{equation*}
 T^{2,0}(X;Y,Z)=\frac{1}{2}\left(T^{2,0}(JX;JY,Z)+\phi^{-1}(T^{1,1}_a)(JX;JY,Z)\right),
\end{equation*}
which yields the first equality of (2). The equivalent formulation is obtained simply by applying $\b$.

(3) is deduced from lemma \ref{lem:adap} and \eqref{eq:etadeta} using the following simple calculation:
\begin{align*}
 T(\xi;X,Y)&=g(\xi,\nabla_X Y-\nabla_Y X-[X,Y])\\
&=-g(\xi,[X,Y])=-\eta([X,Y])=d\eta(X,Y).
\end{align*}

\noindent Using (3) and equation \eqref{eq:T1}, we obtain the following equivalent equations:
\begin{align*}
 g((\lc F)^1 Y,X)&=T(X;\xi,JY)-\tfrac{1}{2}(T(X;\xi,JY)\\
&\phantom{=}+T(\xi;JY,X)-T(JY;\xi,X)),\\
g((\lc F)^1 Y,X)+\tfrac{1}{2}d\eta(JY,X)&=\tfrac{1}{2}\left(g(X,T^1_+(JY)+T^1_-(JY))\right.\\ &\phantom{=}\left.+g(JY,T^1_+(X)+T^1_-(X))\right),\\
g((\lc F)^1 Y,X)+\tfrac{1}{2}d\eta(JY,X)&=g(X,T^1_+(JY)).
\end{align*}
By (F3) of proposition \ref{prop:NF}, we deduce
\begin{align*}
 g(X,T^1_+(Y))&=-g((\lc F)^1 (JY),X)+\tfrac{1}{2}d\eta(Y,X)\\
&=-g(JX,2N^1(JY))-\tfrac{1}{2}d\eta(JY,JX)+\tfrac{1}{2}d\eta(Y,X)\\
&=g(X,2JN^1(JY))\\
&=g(X,-\tfrac{1}{2}J\J Y).
\end{align*}
This yields the result on $T^1_+$ in (4). Concerning $T^1_-$, we use (3) and the fact that $d\eta(\cdot,J\cdot)=-d\eta(J\cdot,\cdot)$, 
to reduce \eqref{eq:T2} to
\begin{equation*}
 -\tfrac{3}{2}((\b T)(\xi;Y,JZ)+(\b T)(\xi;JY,Z))=0,
\end{equation*}
which, by definition of $\b$ is equivalent to
\begin{align*}
 T(\xi;Y,JZ)+T(Y;JZ,\xi)+T(JZ;\xi,Y)&\\+T(\xi;JY,Z)+T(JY;Z,\xi)+T(Z;\xi,JY)&=0.
\end{align*}
Once more making use of (3) and the above property of $d\eta$, we see that this is equivalent to
\begin{align*}
 -g(Y,T^1_+(JZ)+T^1_-(JZ))+g(JZ,T^1_+(Y)+T^1_-(Y))&\\-g(Z,T^1_+(JY)+T^1_-(JY))+g(JY,T^1_+(Z)+T^1_-(Z))&=0.
\end{align*}
Using the symmetry and skew-symmetry of the respective parts, one finally obtains the equivalent condition
\begin{equation*}
 g(JZ,T^1_-Y)+g(Z,T^1_-(JY))=0,
\end{equation*}
which completes the proof of (4).

\textbf{Fourth step:} We now prove the last claim. By the above arguments and the fact that $\b N^{0,2}=0$, we see that \eqref{eq:Tcont} is fulfilled if we choose $T^{0,2},T^{1,1}_a,T^{2,0}$ according to the conditions above, i.e all other parts of $T_c$ may be chosen freely. Now, assuming $(\b T)^+=\omega$ and $T^{1,1}_s=B$, we see that $\omega=\b(T^{1,1}_a+T^{2,0})$ and obtain
\begin{align*}
\b(T^{2,0})&=\tfrac{1}{2}(\b(T^{1,1}_a+T^{2,0})+\b(T^{2,0}-T^{1,1}_a))\\
&=\tfrac{1}{2}\omega,\\
\b(T^{1,1}_a)&=\tfrac{1}{2}(\b(T^{1,1}_a+T^{2,0})-\b(T^{2,0}-T^{1,1}_a))\\
&=\tfrac{1}{2}\omega.
\end{align*}
Thus, by lemma \eqref{eq:bOzn} and \eqref{eq:bOa}, we can deduce
\begin{align*}
 T^{2,0}=\tfrac{3}{2}(\b T^{2,0}-\M\b T^{2,0})=\tfrac{3}{4}(\omega-\M\omega),\\
 T^{1,1}_a=\tfrac{3}{4}(\b T^{1,1}_a-\M\b T^{1,1}_a)=\tfrac{3}{8}(\omega+\M\omega).
\end{align*}
As we have seen above, equations \eqref{eq:T2} and \eqref{eq:T1} are satisfied iff we choose $T^1_+$ as described above and $T^1_-\in \End_-^J(\Cont)$ and $T^2=d\eta$. The explicit description of $T$ is obtained by putting together all of the above data.
\end{proof}

One might now use this result to construct certain ``canonical connections'', by settig $T^{1,1}_s$, $(\b T)^+$ and $T^1_-$ equal to certain forms geometrically defined on a metric contact manifold.
\begin{Rem}
 Note that, unlike in the hermitian case, the Levi-Civit\` {a} connection is never adapted. If it were, than $T=0$ would have to satisfy the properties of the above theorem. However, $\xi\otimes d\eta$ never vanishes (due to the contact condition $\eta\wedge(d\eta)^m\neq 0$).
\end{Rem}

\begin{Rem} \textbf{The case of a 3-manifold}
 
Using the results of remark \ref{rem:formsthree}, we see that in this case $\omega$ does not appear. Furthermore, any endomorphsim of $\Cont$ commuting with $J$ is locally given by its value on $e_1$ (freely choosable) as its value on $f_1$ is then determined by the commutativity rule.
\end{Rem}

%
\subsection{The (generalized) Tanaka-Webster connection and CR connections}
Assume that $(M,H,J,\eta)$ is a strictly pseudoconvex CR manifold. On such manifolds, one has a canonical choice for the adapted connection, namely the \emph{Tanaka-Webster connection} $\tw$. This connection is defined as the metric connection whose torsion is given by
\begin{align}
T(X,Y)&=L_\eta(JX,Y)\xi,\\
T(\xi, X)&=-\frac{1}{2}([\xi,X]+J[\xi,JX])=-\tfrac{1}{2}J\J X\label{eq:Tanaka2}
\end{align}
for any $X,Y\in\Gamma(\Cont)$.
The part of the torsion given by the second equation is called the \emph{pseudo-hermitian torsion} and denoted $\tau(X)=T(\xi,X)$. We will now describe this connection in terms of the defining data according to theorem \ref{thm:torsion}. We begin by noting that as $JN=0$, we have $T^{0,2}=N^{0,2}=0$. Furthermore, $T(\Cont,\Cont)\subset \Xi$, and therefore, we have to choose $\omega=0,\;B=0$ such that $T^{1,1}=T^{2,0}=0$. The part $T^2=d\eta$ is determined independently of $\nabla$ anyway. Finally, $\tau$ lies in $\Omega^1_+(\Cont,\Cont)$ and thus, $T^1_-=0.$ We summarize our findings on the Tanaka-Webster connection in the following lemma, in which we also characterize its generalization to metric contact manifolds:
\begin{Lem}
 The Tanaka-Webster connection of a strictly pseudoconvex CR manifold is given by the following defining data:
 \begin{align*}
  (\b T)^+=0,\quad
 T^{1,1}_s=0\quad\text{and}\quad
 T^1_-=0.
 \end{align*}
Using the same defining data on a general metric contact manifold, one obtains the generalized Tanaka-Webster connection constructed in \cite[section 3.2]{Ni}.
\end{Lem}
\begin{proof}
 We have already established the first statement and what remains to prove is the second one. To this end, we consider the torsion of that connection given in the aforementioned paper, equation (3.7):
\begin{equation*}
 T=N+\xi\otimes d\eta+\tfrac{1}{4}\eta\wedge d\eta+\tfrac{1}{4}\eta\wedge(J-J\J),
\end{equation*}
where the differences between the forumla noted here and the one in \cite{Ni} are due to different conventions (namely for $N$ and for the wedge product of one-forms with endomorphisms). Noting that $\eta\wedge J=\eta\wedge d\eta-\xi\otimes d\eta$, we obtain
\begin{align*}
 T&=N+\xi\otimes d\eta +\tfrac{1}{4}\eta\wedge d\eta-\tfrac{1}{4}\eta\wedge(J\J)+\tfrac{1}{4}(\eta\wedge d\eta-\xi\otimes d\eta)\\
&=N^{0,2}+\xi\otimes d\eta-\tfrac{1}{2}\eta\wedge (J\J).
\end{align*}
This is the torsion of an adapted connection where all freely choosable parts are equal to zero.
\end{proof}
Thus, the generalization of the Tanaka-Webster connection constructed by Nicolaescu is a very natural one. Note that the only difference between the CR Tanaka-Webster connection and the generalized one is the part $T^{0,2}=N^{0,2}$, which vanishes if the manifold is CR. Note also that it is precisely this part in which this generalization differs from the one constructed by Tanno (cf \cite[Prop 3.1]{Tanno}, see also \cite[Prop 2.1]{Petit} for a slightly different description) and which ensures that our (or Nicolaescu's) connection is adapted.\vspace{10pt}

Using the complex description of a CR structure, one has an involutive space $\Cont^{1,0}$. One may now ask oneself whether there are adapted connections that are torsion-free on this space.
\begin{Def}[{\cite[p. 369]{Ni}}]
 An adapted connection on a CR manifold is called a \emph{CR connection} if its torsion (extended $\C$-linearly to $\Cont_c$) satisfies
\begin{equation*}
 T(\Cont^{1,0},\Cont^{1,0})=0.
\end{equation*}
\end{Def}
An easy calculation shows that $\tw$ is of that type and thus, this class is nonempty. In fact, using theorem \ref{thm:torsion}, we may give a complete description of this class:
\begin{Lem}
 An adpated connection $\nabla(\omega, B, \Phi)$ is CR if and only if $\omega=0$.
\end{Lem}
\begin{proof}
 The space $\Cont^{1,0}$ is given by elements of type $X-iJX$, where $X\in\Cont$. Thus, we obtain the condition
\begin{equation*}
 0=T(X-iJX,Y-iJY)=T(X,Y)-T(JX,JY)-i(T(JX,Y)+T(X,JY)),
\end{equation*}
which, because $T(X,Y)$ is a real vector, is equivalent to
\begin{equation*}
 T(X,Y)=T(JX,JY)\quad\text{and}\quad T(JX,Y)=-T(X,JY).
\end{equation*}
We only need to satisfy the first condition as it implies the second one. This first condition implies that $T^{2,0}$ and $T^{0,2}$ as well as $T^2_-$ must vanish. Both $T^2_-$ and $T^{0,2}$ vanish anyway, so we obtain the condition $T^{2,0}=0$. From the proof of theorem \ref{thm:torsion} we know that $T^{2,0}=\tfrac{3}{4}(\omega-\M\omega)$. This yields $\omega=\M\omega$. However, this would mean that $\omega=\b \omega=\b\M\omega=\tfrac{1}{3}\omega$, which is absurd and thus $\omega=0$.
\end{proof}

\section{Dirac operators}
\subsection{Spinor connections and geometric Dirac operators}

In this section, we consider a metric contact manifold with a fixed $Spin$ or $Spin^c$ structure. Associated to this structure is always the spinor bundle, which we shall denote $\S$ in the case of a $Spin$ structure and $\S^c$ in the case of $Spin^c$ structure. Recall that any metric contact manifold always admits a canonical $Spin^c$ structure, whose spinor bundle is given by $\S^c=\Lambda^{0,\smallast} \Cont^*$ (cf. \cite{Petit}).

Now, every metric connection $\nabla$ on $TM$ induces a connection on the spinor bundles: $\nabla$ induces a connections form $C$ on the bundle of orthonormal, oriented frames $P_{SO}(M)$, locally given by
\begin{equation*}
 C^s(X)=C(ds(X))=\tfrac{1}{2}\sum_{i<j}g(\nabla_X s_i,s_j)E_{ij},
\end{equation*}
where $s=(s_1,...,s_n)$ is a local section of $P_{SO}(M)$ and $E_{ij}$ is the $n\times n$ matrix with entries $(E_{ij})_{ij}=-1$, $(E_{ij})_{ji}=1$ and all others equal to zero. This form then lifts to one on the $Spin$ or $Spin^c$ principal bundle such that the following diagrams commute respectively:
\begin{equation*}
 \xymatrix{
TP_{Spin}(M)\ar[d]\ar[r]^{\widetilde{C}}&\mathfrak{spin}_n\ar[d]\\
TP_{SO}(M)\ar[r]^{C}&\mathfrak{so}_n
},
\end{equation*}
\begin{equation*}
 \xymatrix{
TP_{Spin^c}(M)\ar[d]\ar[r]^{\widetilde{C}^Z}&\mathfrak{spin^c}_n\simeq \mathfrak{spin}_n\oplus i\R\ar[d]\\
TP_{SO}(M)\times P_1\ar[r]^{C\times Z}&\mathfrak{so}_n\oplus i\R
},
\end{equation*}
where the vertical arrows denote the respective two-fold converings and $Z$ is a connection form on the $U_1$-bundle $P_1$ associated to any $Spin^c$ strcuture via $P_1=\faktor{P_{Spin^c}}{Spin}$.
Then, this connection form induces a connection $\widetilde{\nabla}$ on the associated vector bundle $\S=P_{Spin}\times_\kappa \Delta_n$ (resp. $\widetilde{\nabla}^Z$ on $\S^c=P_{Spin^c}\times_{\kappa^c}\Delta_n$), where $\kappa:Spin_n\to \Delta_n$ is the spinor representation and $\kappa^c$ is the complex-linear extension to $Spin^c$. We can locally describe $\widetilde{\nabla}$ by
\begin{equation}\label{eq:spinconn} 
 \widetilde{\nabla}_X \phi|_{U}=[\widetilde{s},X(v)+\sum_{i<j}g(\nabla_X s_i,s_j)s_i\cdot s_j\cdot v],
\end{equation}
where $\phi|_U=[\widetilde{s},v]$ with $s=(s_1,...,s_n)\in\Gamma(U,P_{SO}(M)), v\in C^{\infty}(U,\Delta_n)$ and $\widetilde{s}$ is a lifting of $s$ to $P_{Spin}(M)$, and $\widetilde{\nabla}^Z$ by
\begin{equation}\label{eq:spincconn}
(\nsc_X\phi)(x)=[\widetilde{s\times e},X(v)+\frac{1}{2}\sum_{j<k}g(\nabla_X s_j,s_k)s_j\cdot s_k\cdot v+\frac{1}{2}Z\circ de (X).v],
\end{equation}
where $\phi|_U=[\widetilde{s\times e},v]$ with $s\in\Gamma(U,P_{SO})$ and $e\in\Gamma(U,P_1)$, $\widetilde{s\times e}$ is a lifting to $\Gamma(U,P_{Spin^c})$ and finally, $v\in C^{\infty}(U,\Delta_n)$.

Then, both spinorial connections have the following well-known properties:
\begin{Lem}
Let $X,Y$ be vector fields and  let $\varphi\in\Gamma(\S)$. Then for the connection $\widetilde{\nabla}$ induced on $\S$ by any metric connection $\nabla$ on $TM$, we have
\begin{equation}\label{eq:conncliff}
\widetilde{\nabla}_X(Y\cdot\varphi)=(\nabla_X Y)\cdot\varphi+Y\cdot\widetilde{\nabla}_X\varphi.
\end{equation}
Furthermore, $\widetilde{\nabla}$ is metric with respect to the hermitian scalar product on $\S$. Analogous results hold for $\widetilde{\nabla}^Z$.
\end{Lem}

Note that the connection used on $TM$ in \eqref{eq:conncliff} is $\nabla$ and not $\nabla^g$, as one usually demands on a Dirac bundle.

Now, each connection on $\S$ or $\S^c$ induces a first order differential operator:
\begin{Def}
 Let $M$ be spin. The first order differential operator
\begin{equation*}
 \D(\nabla)\colon \Gamma(\S)\xrightarrow[]{\widetilde{\nabla}}\Gamma(T^*M\otimes \S)\xrightarrow[]{c}\Gamma(\S),
\end{equation*}
where $c$ denotes Clifford multiplication, is called the \emph{Dirac operator} associated to $\nabla$.

\noindent Let $M$ be $Spin^c$. The first order differential operator
\begin{equation*}
 \D_Z(\nabla)\colon \Gamma(\S^c)\xrightarrow[]{\widetilde{\nabla}^Z}\Gamma(T^*M\otimes \S^c)\xrightarrow[]{c}\Gamma(\S^c)
\end{equation*}
is called the \emph{Dirac operator} associated to $\nabla$ and $Z$.
\end{Def}

Using the local formulas \eqref{eq:spinconn} and \eqref{eq:spincconn}, one deduces the following results stated in \cite{Ni}:
\begin{Lem}\label{lem:Diraccomp}
 Let $\nabla$ be a metric connection on $M$ and let $M$ be spin or spin${}^c$ respectively. Then, the following identities are satisfied by the Dirac operators of $\nabla=\lc+A$:
\begin{align}
 \D(\nabla)&=\D(\lc)-\frac{1}{2}c(\tr A)+\frac{3}{2}c(\b A),\label{eq:diraccomp}\\
\D_Z(\nabla)&=\D_Z(\lc)-\frac{1}{2}c(\tr A)+\frac{3}{2}c(\b A).\label{eq:diraccompc}
\end{align}
\end{Lem}

Next, we prove that all Dirac operators defined above mimick the following behaviour of $\D(\lc)$:
\begin{Lem}
 The Dirac operator $\D_Z(\nabla)$ (or $\D(\nabla)$) of any metric connection $\nabla$ satisfies
\begin{equation*}
 \D_Z(\nabla)(f\phi)=\grad f\cdot \phi+f\D_Z(\nabla)\phi
\end{equation*}
 for any $f\in C^\infty(M)$ and $\phi\in \Gamma(\S^c)$.
\end{Lem}
\begin{proof}
 The identity is satisfied by $\D_Z(\lc)$. The extension to any Dirac operator  $\D_Z(\nabla)$ then follows from \eqref{eq:diraccompc}, because Clifford multiplication is defined pointwise and thus commutes with the multiplication by $f$.
\end{proof}

\begin{Def}
 A connection $\nabla$ is called \emph{nice} if its Dirac operator $\D(\nabla)$ is symmetric. Two connections $\nabla^1,\nabla^2$ are called \emph{Dirac equivalent} if they induce the same Dirac operator: $\D(\nabla^1)=\D(\nabla^2)$.
\end{Def}
Now, using that the Dirac operators $\D(\lc)$ and $\D_Z(\lc)$ are symmetric, one deduces the following results, also stated in \cite{Ni}
\begin{Lem}
 \begin{numlist}
  \item The connection $\nabla$ is nice if and only if the torsion of $\nabla$ satisfies $\tr T=0$. This is the case if and only if the $Spin^c$ Dirac operator $\D_Z(\nabla)$ is symmetric for any $U_1$ connection $Z$.
  \item Let $\nabla^1$ and $\nabla^2$ be nice metric connections. Then they are Dirac equivalent if and only if $\b T^1=\b T^2$. This holds holds if and only if $\D_Z(\nabla^1)=\D_Z(\nabla^2)$ for any $U_1$ connection $Z$.

 \item Let $(M,g)$ be complete. Then any Dirac operator $\D(\nabla)$ or $\D_Z(\nabla)$  that is symmetric is indeed essentially self-adjoint.
 \end{numlist}
\end{Lem}
\begin{proof}
 The first two statements are immediately derived from the Lemma \ref{lem:Diraccomp}. For the essential self-adjointness of the Dirac operator, consider the proof of essential self-adjointness of the Dirac operator of Wolf,\cite{WolfESA}. The proof is given only for $\D(\lc)$, but is indeed extendable to all symmetric Dirac operators: The domain of $\D(\nabla)$ is $\Gamma_c(\S^c)\subset L^2(\S^c)$ (we consider the $Spin^c$ case here, the arguments in the $Spin$ cases are analogous), where the subscript $c$ denotes compact support. The proof uses the following norm on the domain of $\D_Z(\nabla)^*$, where ${}^*$ denotes the adjoint:
\begin{equation*}
 N(\phi)=\sqrt{\|\phi\|^2_{L^2}+\|D_Z(\nabla)^*\|^2_{L_2}}.
\end{equation*}
Then the following results are proven:
\begin{enumerate}
 \item If $\Gamma_c(\S^c)$ is dense in $\dom(\D_Z(\nabla)^*)$ with respect to the $N$ norm, then $\D_Z(\nabla)^*$ is essentially self-adjoint.
 \item $\Gamma_c(\S^c)$ is dense in $\dom_c(\D_Z(\nabla)^*)$ with respect to the $N$ norm.

 \item If $(M,g)$ is complete, then $\dom_c(\D_Z(\nabla)^*)$ is dense in $\dom(\D_Z(\nabla)^*)$ with respect to the $N$ norm.
\end{enumerate}
The proof of the first statement requires nothing of $D_Z(\nabla)$ but to be a closable operator. To prove the second one, we only need $\D_Z(\nabla)$ to be an elliptic differential operator of order one, which it is because it differs form $\D_Z(\lc)$ only by lower oder terms.
Finally, the proof of the third statement does not make use of the explicit form of $\D_Z(\nabla)$ either, it only needs it to fulfil $\D_Z(\nabla)(f\phi)=\grad f\cdot \phi+f\D_Z(\nabla)\phi$ for any $f\in C^\infty(M)$ and $\phi\in \Gamma(\S^c)$, which it does.
\end{proof}

In the following section, we will use these formulas to determine which adapted connections induce symmetric Dirac operators and which induce the same Dirac operators.

\subsection{Dirac operators of adapted connections}

We now use the formulas in the preceding section to establish some properties of Dirac operators associated to adapted connections. To this end, we calculate the trace of the torsion of an adapted connection and its image under the Bianchi operator. Recall (equation \eqref{eq:torsion}) that the torsion of such a connection is given by
\begin{equation*}
 T=N^{0,2}+\tfrac{9}{8}\omega-\tfrac{3}{8}\M\omega+B+\xi\otimes d\eta-\tfrac{1}{2}\eta\wedge (J\J) +\eta\wedge \Phi,
\end{equation*}
where we can freely choose
\begin{align*}
\omega&\in\Op=\{\omega\in\Otc\;|\;\omega=\omega^{2,1}+\omega^{1,2}\text{ as (p,q)-forms}\},\\
B&\in \Os=\{D\in\Omega^2(\Cont,\Cont)\;|\;D(J\cdot,\cdot)=D\text{ and } \b D=0\},\\
\Phi&\in\End^J_-(\Cont)=\{F:\Cont\to\Cont\;|\;g(X,FY)=-g(FX,Y)\text{ and }FJ=JF\}.
\end{align*}
We know that $\omega$ and $N^{0,2}$ are traceless and so is $\eta\wedge \Phi$ because $\Phi$ is skew-symmetric. Furthermore, we immediately see that $\tr(\eta\wedge(J\J))(X)=0$ and we calculate for some adapted basis $(e_i,f_i)$, making use of the various properties of $\J$:
\begin{align*}
 \tr(\eta\wedge(J\J))(\xi)&=\sum_{i=1}^m (\eta\wedge J\J)(e_i,e_i,\xi)+(\eta\wedge J\J)(f_i,f_i,\xi)\\
&=\sum_{i=1}^m g(e_i,J\J e_i)+g(f_i,J\J f_i)\\
&=0.
 \end{align*}
Thus, we deduce
\begin{equation}\label{eq:trT}
 \tr T=-\tfrac{3}{8}\tr\M\omega+\tr B.
\end{equation}
Concerning the Bianchi operator, we note that $\b B=0$ and $\b N^{0,2}=0$. Furthermore, because $J\J$ is symmetric, $\b(\eta\wedge J\J)=0$. Then, using \eqref{eq:bM}, we deduce
\begin{equation}\label{eq:bT}
 \b T=\omega+\tfrac{1}{3}\eta\wedge d\eta+\b(\eta\wedge\Phi).
\end{equation}

Using these equations, we deduce the following result:
\begin{Prop}\label{prop:Dirac}
 \begin{numlist}
  \item The  adapted connection $\nabla(\omega,B,\Phi)$ is nice if and only if $\tr (B)=\tfrac{3}{8}\tr\M \omega$. Moreover, if $(M,g)$ is complete, the Dirac operator of any such connection is essentially self-adjoint.

 \item Two nice adapted connections $\nabla(\omega,B,\Phi)$ and $\nabla(\hat\omega,\hat B,\hat \Phi)$ are Dirac equivalent if and only if $\omega=\hat \omega$ and $\Phi=\hat\Phi$.

 Thus any Dirac equivalence class of nice adapted connections is determined by $\omega, \Phi$, while the connections in it are parametrized by $B\in \Os$ such that $\tr B=\tfrac{3}{8}\tr\M \omega$.
 \end{numlist}
\end{Prop}
\begin{proof} The first part of (1) is obvious from \eqref{eq:trT}. For (2), recall that $\b(\eta\wedge \Phi)$ completely determines $\Phi$ (compare lemma \ref{lem:bEnd}).
\end{proof}

\begin{Rem}
In particular, we see that in a Dirac equivalence class of nice connections, there is a CR connection if and only if all connections in this class are CR. Thus, contrary to the claim in the last section of \cite{Ni}, there may be more than one CR connection in a Dirac equivalence class, as $B$ may still be chosen freely as long as it satisfies $\tr B=0$ (due to (1) of the above proposition and $\omega=0$). In fact, the uniqueness proof in the above paper uses that the torsion of a CR connection would satisfy $T(X;Y,Z)=0$ for any $X,Y,Z\in \Cont$, which seems to be wrong.

In fact, on a three-manifold the uniqueness result does hold, because using the results mentioned in remark \ref{rem:formsthree}, we see that there is no non-zero form in $\Omega^2(\Cont,\Cont)$ with vanishing trace. In higher dimesnions, such forms do exist. We note that in dimension three, there is in fact (by the same argument) a unique adapted connection in each Dirac equivalence class and we do not need to demand that the connection be CR in order to obtain uniqueness.
\end{Rem}

We now use these results to characterize some connections that are Dirac equivalent to certain known connections:
\begin{Cor}
 An adpated connection $\nabla(\omega,B,\Phi)$ is Dirac equivalent to the generalized Tanaka-Webster connection if and only if it satisfies $\omega=0$, $\Phi=0$ and $\tr B=0$. Any such connection is CR and its Dirac operator takes the form
\begin{equation*}
 \D_Z(\nabla)=\D_Z(\lc)+\tfrac{1}{4}c(\eta\wedge d\eta).
\end{equation*}
\end{Cor}
\begin{proof}
 To induce the same Dirac operator as $\tw$, the connection will need to be nice. Thus, by proposition \ref{prop:Dirac}, $\tr B=\tfrac{3}{8}\tr\M \omega$. Because all freely choosable parts of $T^{TW}$ vanish, we obtain, again by proposition \ref{prop:Dirac}, that $\omega$ and $\Phi$ vanish, which in turn implies $\tr B=0$. The explicit fomrula is immediately deduced from the above calculations of $\b T$ and from \eqref{eq:diraccompc}.
\end{proof}

\begin{Cor}
 An adpated connection $\nabla(\omega,B,\Phi)$ is Dirac equivalent to the Levi-Civit\`{a} connections if and only if it satisfies $\omega=0$, $\tr B=0$ and $\Phi=-\frac{1}{2}J$. Any such connections is CR.
\end{Cor}
\begin{proof}
 Again, the connection will need to be nice, i.e. $\tr B=\tfrac{3}{8}\tr\M \omega$. Now, for the comparison with $\lc$ we cannot use proposition \ref{prop:Dirac} as $\lc$ is not adapted. Instead, using \eqref{eq:bT}, we deduce the condition $0=\omega+\tfrac{1}{3}\eta\wedge d\eta+\b(\eta\wedge\Phi)$. As they take their arguments form different spaces, $\omega$ and $\tfrac{1}{3}\eta\wedge d\eta+\b(\eta\wedge\Phi)$ will have to vanish seperately. We calculate
\begin{equation*}
 \tfrac{1}{3}(\eta\wedge d\eta)(\xi,X,Y)=-\b(\eta\wedge \Phi)(\xi,X,Y)\quad\Leftrightarrow \quad d\eta(X,Y)=-2g(Y,\Phi X).
\end{equation*}
Using that $d\eta=g(J\cdot,\cdot)$ then yields then claim.
\end{proof}
 Note that we have just proven that adapted connections may induce the same Dirac operators as non-adapted ones.

\bibliographystyle{../halpha}
\bibliography{../bibli}
\end{document}